\documentclass[11pt,reqno,a4paper]{amsart}
\usepackage[english]{babel}
\usepackage{enumitem}
\usepackage[all]{xy}
\SelectTips{cm}{11}

\makeatletter
\newlength{\@thlabel@width}%
\newcommand{\thmenumhspace}{\settowidth{\@thlabel@width}{\upshape(i)}\sbox{\@labels}{\unhbox\@labels\hspace{\dimexpr-\leftmargin+\labelsep+\@thlabel@width-\itemindent}}}
\makeatother

\title{On the higher Whitehead product}

\author{Marek Golasi\'nski}
\address{Faculty of Mathematics and Computer Science, University of Warmia and Mazury,
S\l oneczna 54 Street, 10-710 Olsztyn, Poland}
\email{marekg@matman.uwm.edu.pl}

\author[Thiago de Melo]{Thiago de Melo\textsuperscript{*}}
\address{Instituto de Geoci\^encias e Ci\^encias Exatas\\ UNESP--Univ Estadual Paulista\\ Av.~24A, 1515, Bela Vista. CEP 13.506--900. Rio Claro--SP, Brazil}
\email{tmelo@rc.unesp.br}
\thanks{\textsuperscript{*}Supported by FAPESP 2014/21926-1}


\theoremstyle{definition}
\newtheorem{definition}{Definition}[section]
\theoremstyle{plain}
\newtheorem{theorem}[definition]{Theorem}
\newtheorem{proposition}[definition]{Proposition}
\newtheorem{corollary}[definition]{Corollary}
\newtheorem{lemma}[definition]{Lemma}

\theoremstyle{remark}
\newtheorem{remark}[definition]{Remark}

\begin{document}
\baselineskip=1.5em

\begin{abstract} G.\ J.\ Porter's approach \cite{porter} is used to derive some properties of higher order Whitehead products, similar to those ones
for triple products obtained by Hardie in \cite{hardie}.  Computations concerning the higher order Whitehead product for spheres and projective spaces are presented as well.
\end{abstract}

\maketitle

\section*{Introduction}
Whitehead products play an important role in algebraic topology and its applications.
The classical Whitehead product $[f,g]$ of (homotopy classes of) maps $f:\mathbb{S}^m\to X$ and $g:\mathbb{S}^n\to X$ is (the homotopy class)
represented by a map $h:\mathbb{S}^{m+n-1}\to X$ and  defined by means of the so called Whitehead map
$\omega:\mathbb{S}^{m+n-1}\to \mathbb{S}^m\vee \mathbb{S}^n$, the attaching map for the $(m+n)$-cell of $\mathbb{S}^m\times \mathbb{S}^n$.

Arkowitz \cite{ark}  constructed a generalization $\omega:\Sigma(A\wedge B)\to \Sigma A\vee \Sigma B$ of the Whitehead map
to define the generalized Whitehead product $[f,g]$ of maps $f :\Sigma A\to X$ and $g :\Sigma B\to X$.
Then Zeeman \cite{massey} and Hardie \cite{hardie} generalized the Whitehead product in a different context, that is, they defined
a triple spherical Whitehead product $[f_1,f_2,f_3]$ of maps $f_i:\mathbb{S}^{m_i}\to X$ for $i=1,2,3$. The main result from \cite{hardie} deals with
the $r$\textsuperscript{th} order spherical Whitehead product $[f_1,\ldots,f_r]$ for maps $f_i : \mathbb{S}^{m_i}\to X$ with $r\geq 3$ defined in \cite{hardie1},
and in particular, it states that the triple product $[f_1,f_2,f_3]$ is a coset of a subgroup of $\pi_{m-1}(X)$, where $m=m_1+m_2+m_3$. Many properties which hold for the classical Whitehead product still hold for the triple one as well.

G.\ J.\ Porter's approach \cite{porter} was to construct the Whitehead map $\omega_r$ for more than two spaces
(see Equation~\eqref{eq.whitehead.map.generalized}).
Hardie's construction from \cite{hardie1} was generalized  in \cite{porter}, where the $r$\textsuperscript{th} order generalized Whitehead product of maps $f_i:\Sigma A_i\to X$ for $i=1,\dots ,r$,
with $r\geq 2$ was introduced. Then the $2^\text{nd}$ order Whitehead product coincides with the generalized Whitehead product
studied by Arkowitz in \cite{ark}.

Higher order Whitehead products are secondary, tertiary, etc.\ analogues of ordinary Whitehead products. They
first appeared in the late 1960's as a part of a research theme studying higher products, or simply higher
structures, in homotopy theory. They are vital for understanding the homotopy theory of certain
basic objects, such as an iterated product of spaces, and their maps into other spaces. Recently, higher Whitehead products
have re-emerged as key players in the homotopy theory of polyhedral products. These are important
objects in toric topology and are being increasingly used in geometric group theory and graph theory.
Porter's construction is very useful in many mathematical constructions. For example, given a simplicial complex $K$ on $n$ vertices,
Davis and Januszkiewicz \cite{dj}  associated two fundamental objects of toric topology: the moment--angle complex
$\mathcal{Z}_K$ and the Davis--Januszkiewicz space $DJ_K$. The homotopy fibration sequence
\[\mathcal{Z}_K\xrightarrow{\tilde{\omega}}DJ_K\to\textstyle\prod\limits^n_{i=1}\mathbb{C}P^\infty\] and its generalization have been studied
in \cite{gt} and \cite{kish}, respectively to show that $\tilde{\omega} : \mathcal{Z}_K\to DJ_K$ is a sum of higher and iterated Whitehead products
for appropriate complexes $K$.

Next, let $\mathcal{F}(\mathbb{R}^{n+1},m)$ be the Euclidean ordered configuration space.
By  Salvatore \cite[Theorem~7]{salvatore}, the homotopy type of $\mathcal{F}(\mathbb{R}^{n+1},m)$ for $n \geq 2$
admits a minimal cellular model
\[\ast=X_0\subseteq X_n\subseteq X_{2n}\subseteq\dots \subseteq X_{mn}\]
whose cells are attached via higher order Whitehead products.

In this paper, we work with triple and $r$\textsuperscript{th} order Whitehead products. The aim of Section~\ref{sec.preliminaries} is to fix some notations,
recall definitions and necessary results from \cite{ando,ark} and  present  properties on separation elements, and the relative generalized Whitehead
product as well. Section~\ref{sec.HOGWP} expounds the main facts from \cite{porter} on $r$\textsuperscript{th} Whitehead products which are used
to discuss the main result from \cite{hardie}. In Section~\ref{sec.triple}, based on some relations from \cite{toda}, we prove
Proposition~\ref{prop.order.15}, the main result of that section, which deals with the non-triviality problem of the triple
product $[\eta_4,\eta_4^2,2\iota_4]\subseteq \pi_{14}(\mathbb{S}^4)$, stated by Hardie in \cite[Section~5]{hardie}. The main result of Section~\ref{sec.main} is Theorem~\ref{teo.main} which extends the result \cite[Theorem~4.3]{hardie} to higher order Whitehead products. Finally, Section~\ref{sec.five} is devoted to some computations concerning the $r$\textsuperscript{th} order Whitehead product for spheres and projective spaces. In particular, Proposition~\ref{wupro} shows that there are spaces $X$ such that the $r$\textsuperscript{th} order generalized Whitehead product $[0_1,\dots,0_r]$
contains a non-trivial element provided $0_i : \mathbb{S}^{m_i}\to X$ are trivial for $i=1,\dots,r$ with $r\ge 4$.

\subsubsection*{Acknowledgement} The authors are deeply grateful to the referee for a number of invaluable suggestions to work on the final version,
especially on Proposition \ref{prop.order.15}. 
They would like to thank Martin Arkowitz for a very helpful discussion and indicating the reference \cite{ando} as well.

\section{Preliminaries}\label{sec.preliminaries}
Denote by $\pi(X,Y)$ the set of (based) homotopy classes of (based) maps $X\to Y$. In the sequel we do not distinguish between a map
and its homotopy class. Write $C(-)$ and $\Sigma(-)$ for the reduced cone and the reduced suspension functors, respectively. As in Toda \cite{toda},
$\iota_n:\mathbb{S}^n\to \mathbb{S}^n$ denotes the identity map of the  $n$-sphere $\mathbb{S}^n$, and
$\eta_n=\Sigma^{n-2}\eta_2:\mathbb{S}^{n+1}\to \mathbb{S}^n$ for $n\geq 2$,  is the iterated suspension of the Hopf map
$\eta_2:\mathbb{S}^3\to \mathbb{S}^2$. We also use freely other symbols from \cite{toda} and write $\iota_X$ for the identity map on a space $X$.

\subsection{Relative generalized Whitehead product}\label{sec.RelGWP} 
Let $A$ and $B$ be spaces. Given maps $f:\Sigma A\to X$ and
$g: \Sigma B \to X$, Arkowitz constructed in \cite{ark} \textit{the Whitehead map}  \begin{equation}\label{eq.whitehead.map.classical}
\omega:\Sigma (A\wedge B) \to \Sigma A\vee \Sigma B
\end{equation} to define  the generalized Whitehead product $[f,g]$ as the map $\nabla (f\vee g) \omega:\Sigma(A\wedge B)\to X$,
where  $\nabla : X\vee X \to X$ is the folding map and $A\wedge B$ is the smash product.
Also, from \cite[Proposition~3.3]{ark}, the product is anti-commutative, that is, $[g,f]=-(\Sigma\sigma)^*[f,g]$, where $\sigma : B\wedge A\to A\wedge B$ is the homeomorphism determined by the  map $B\times A\to A\times B$.

In \cite{ando}, Ando defined the \textit{relative generalized Whitehead product}, briefly described as follows. Let $i_Z:Z\hookrightarrow CZ$ be
the inclusion map. Given $k:X\to Y$, we say that $(h_1,h_2):i_Z\to k$ is a \textit{pair-map} provided the diagram
\[\xymatrix{Z\ar[r]^{h_2} \ar@{^{(}->}[d]_{i_Z} & X \ar[d]^{k} \\ CZ \ar[r]^{h_1} & Y } \] is commutative. Write $\pi_1(Z,k)$ for the set of
homotopy classes of pair-maps $(h_1,h_2):i_Z\to k$.

Note that if $Z=\Sigma Z'$  then the standard co-$H$-structure $\nu : \Sigma Z'\to \Sigma Z'\vee \Sigma Z'$
leads to a pair-map
\[\xymatrix@C=1.5cm{\Sigma Z'\ar[r]^-{\nu}  \ar@{^{(}->}[d]_{i_{\Sigma Z'}} & \Sigma Z'\vee \Sigma Z' \ar@{^{(}->}[d]^{i_{\Sigma Z'}\vee i_{\Sigma Z'}} \\
C(\Sigma Z') \ar[r]^-{C\nu} & C(\Sigma Z')\vee C(\Sigma Z') \rlap{,} } \]
which determines a group structure on $\pi_1(Z,k)$. Moreover, this group is abelian provided $Z'=\Sigma Z''$.
In particular, if $k:X_0\hookrightarrow X$ is an inclusion map and $Z=\mathbb{S}^n$, we obtain the ordinary relative homotopy groups
$\pi_1(\mathbb{S}^n,k)=\pi_{n+1}(X,X_0)=[(\mathbb{D}^{n+1},\mathbb{S}^n),(X,X_0)]$, where $\mathbb{D}^{n+1}=C\mathbb{S}^n$ is the $(n+1)$-disc.

Further, the map $k:X\to Y$ implies  the dual Puppe long exact sequence
\begin{equation}\label{eq.seq.puppe}
\dots \to \pi(\Sigma^2Z,X) \xrightarrow{k_\ast} \pi(\Sigma^2Z,Y) \xrightarrow{j_\ast}  \pi_1(\Sigma Z,k)  \xrightarrow{\delta_\ast} \pi(\Sigma Z,X) \xrightarrow{k_\ast} \dotsb.
\end{equation}

For $k=i_{\Sigma A}\vee \iota_{\Sigma B}$, according to \cite[Section~2(b)]{ando}, there is a commutative diagram
\[
\xymatrix@C=.7cm{ \pi_1(\Sigma Z,i_{\Sigma A}\vee \iota_{\Sigma B}) \ar[r]^-{\delta_\ast} &
\pi(\Sigma Z,\Sigma A\vee \Sigma B) \ar[rr]^-{(i_{\Sigma A}\vee \iota_{\Sigma B})_*} \ar[d]_{\kappa_1} && \pi(\Sigma Z,C(\Sigma A)\vee \Sigma B) \ar[d]^{\kappa_2} \\
& \pi(\Sigma Z,\Sigma A\times \Sigma B) \ar[rr]^-{(i_{\Sigma A}\times \iota_{\Sigma B})_*} & &
\pi(\Sigma Z,C(\Sigma A)\times \Sigma B) \rlap{,} }
\]
where the vertical maps $\kappa_i$ for $i=1,2$ are induced by the inclusions of the wedge into the product, and the top row is exact.

By~\cite[Theorem~(2.3)]{porter}, $(i_{\Sigma A}\times \iota_{\Sigma B})_*\kappa_1(\omega)=0$ and by the commutativity of the diagram above, and
the fact that $\kappa_2$ is isomorphism, $(i_{\Sigma A}\vee \iota_{\Sigma B})_*(\omega)=0$. Note that
$(i_{\Sigma A}\vee \iota_{\Sigma B})_\ast : \pi(\Sigma Z,\Sigma A\vee \Sigma B)\to \pi(\Sigma Z,C(\Sigma A) \vee \Sigma B)$  is an epimorphism, and so $\delta_*$ is a monomorphism. Thus, by the exactness of the top row, there is a unique element
$\overline{\omega}\in \pi_1(\Sigma Z,i_{\Sigma A}\vee \iota_{\Sigma B})$ such that $\delta_\ast(\overline{\omega})=\omega$.

Let $k:X\to Y$ and $Z=A\wedge B$ in the above. Then, given $f=(f_1,f_2)\in \pi_1(\Sigma A,k)$ and $g\in \pi(\Sigma B,X)$, the element $\overline{\omega}$ determines a pair-map $(\omega',\omega)$ commuting the diagram
\[ \xymatrix@C=1.5cm{
\Sigma(A\wedge B) \ar[r]^-{\omega} \ar@{^{(}->}[d]_{i_{\Sigma(A\wedge B)}} & \Sigma A\vee \Sigma B \ar@{^{(}->}[d]^{i_{\Sigma A}\vee \iota_{\Sigma B}} \ar[r]^-{f_2 \vee g} & X\vee X \ar[d]^{k\vee k}\ar[r]^-{\nabla} & X \ar[d]^{k} \\
C\Sigma(A\wedge B) \ar[r]^-{\omega'} & C(\Sigma A)\vee \Sigma B \ar[r]^-{f_1 \vee kg} & Y\vee Y \ar[r]^-{\nabla} & Y \rlap{.}
} \]
As in \cite{ando}, the
\textit{relative generalized Whitehead product} $[f,g]_R\in \pi_1(\Sigma(A\wedge B),k)$ is defined by the pair-map \begin{equation*}\label{eq.RelGWP}
\bigl(\nabla(f_1\vee kg),\nabla(f_2\vee g)\bigr)_*(\overline{\omega})=\bigl(\nabla(f_1\vee kg)\omega',\nabla(f_2\vee g)\omega\bigr).
\end{equation*}
Also therein, Ando compares his construction to the one introduced by Hardie in \cite{hardie}. We describe that here to obtain
Proposition~\ref{prop.separation}.

First, since the generalized Whitehead product $[0,g]$ is trivial and the product $\Sigma A\times \Sigma B$ is a push-out,
we obtain a map $\lambda_A(g):C\Sigma(A\wedge B)\to X$ commuting the diagram
\[
\xymatrix{
\Sigma(A\wedge B) \ar@{^{(}->}[d] \ar[r]^{\omega} & \Sigma A\vee \Sigma B \ar[rr]^-{\nabla(0\vee g)} \ar@{^{(}->}[d] & & X \ar@{=}[d] \\
C\Sigma(A\wedge B)  \ar[r]\ar@/_3ex/[rrr]_(.5){\lambda_A(g)} & \Sigma A\times\Sigma B \ar[r]^-{p_B} &\Sigma B \ar[r]^{g} & X \rlap{,}
}
\]
where $p_B:\Sigma A\times \Sigma B\to \Sigma B$ is the projection map.

Let $[f_1,kg]':C\Sigma(A\wedge B)\to Y$ be given by
\begin{equation*}
[f_1,kg]' [y,t]=\begin{cases}
\lambda_A(kg)[y,2t], & 0\leq t\leq \frac{1}{2}, \\[1ex]
[f_1\sigma'_{2t-1},kg](y), & \frac{1}{2} \leq t \leq 1,
\end{cases}
\end{equation*}
where $y\in \Sigma(A\wedge B)$, and $\sigma'_{2t-1}:\Sigma A\to C\Sigma A$ is given by $\sigma'_{2t-1}(x)=[x,2t-1]$ for $x\in \Sigma A$. Thus, by \cite[Theorem~(4.6)]{ando} we obtain
\begin{equation}[f,g]_R = \bigl( [f_1,kg]', [f_2,g] \bigr).\label{eq.trick}\end{equation}
Further,
\begin{equation}\label{eq.double}
\varphi[f_1,kg]'=[\varphi f_1,\varphi kg]'
\end{equation}
for $\varphi : Y\to Z$ and
\[ [(C\Sigma\alpha)^\ast f_1,(\Sigma \beta)^\ast kg]' = (C\Sigma(\alpha\wedge\beta))^\ast [f_1,kg]'\]
for $\alpha : A'\to A$ and $\beta : B'\to B$.

Next, suppose that $A=\Sigma A'$ and consider $[f_2,g]'':\Sigma(A\wedge B)\to X$ given~by
\begin{equation*}
[f_2,g]'' [y,t]=\begin{cases}
\lambda_A(g)[y,4t], & 0\leq t\leq \frac{1}{4}, \\[1ex]
[f_2\sigma_{\frac{1}{2}(4t-1)},g](y), & \frac{1}{4} \leq t \leq \frac{3}{4}, \\[1ex]
\lambda_A(g)[y,4-4t], & \frac{3}{4}\leq t\leq 1,
\end{cases}
\end{equation*}
where $y\in \Sigma(A\wedge B)$, and $\sigma_{\frac{1}{2}(4t-1)}:A\to \Sigma A$ is given by
$\sigma_{\frac{1}{2}(4t-1)}(a)=[a,\frac{1}{2}(4t-1)]$ for $a\in A$. Note that for $t=\frac{1}{4}$, $f_2\sigma_0=0$ and $\lambda_A(g)[y,1]=[0,g](y)$,
and the same for $t=\frac{3}{4}$. So, the map $[f_2,g]''$ is well defined and $[f_2,g]''= [f_2,g]$ (cf.~\cite[Theorem~(2.4)]{hardie}).

\begin{remark} We work with relative Whitehead products for an inclusion map $k:X_0\hookrightarrow X$. It means that  (homotopy classes
of) maps $f: (C\Sigma A,\Sigma A)\to  (X,X_0)$ and $g: \Sigma B  \to X_0$ determine the map of pairs
\begin{equation*}
[f,g]_R: (C\Sigma (A\wedge B),\Sigma (A\wedge B)) \to  (X,X_0).
\end{equation*}
\end{remark}

We finish this section with some properties of the relative product $[-{,}-]_R$.

Let $f,f_i\in \pi_1(\Sigma A,k)$ and $g,g_i\in \pi(\Sigma B,X)$ for $i=1,2$. According to \cite[Proposition~(4.5)]{ando},
if $A$ and $B$ are suspensions then $[f,g_1+g_2]_R=[f,g_1]_R+[f,g_2]_R$ and
$[f_1+f_2,g]_R=[f_1,g]_R+[f_2,g]_R$. Also, \cite[Theorem~(4.8)]{ando} states that
$[g,f]_R=-(C\Sigma\sigma,\Sigma\sigma)^*[f,g]_R$, where $\sigma:B\wedge A\to A\wedge B$ is the canonical homeomorphism.

\subsection{Generalized separation element}\label{sec.sep.map}
Given maps $f,g: CA \to X$ such that $f_{\vert A}=g_{\vert A}$, following James
\cite[Section~10]{james} and Tsuchida \cite[Section~3]{TK}, we define the \textit{generalized separation element} as (the homotopy
class of) the map $d(f,g):\Sigma A \to X$ given by:
\begin{equation*}
d(f,g)[a,t]=\begin{cases}
f[a,2t], & 0\leq t\leq \frac{1}{2}, \\[1ex]
g[a,2-2t], & \frac{1}{2}\leq t\leq 1.
\end{cases}
\end{equation*}
Further, if $A=\Sigma A'$ then we consider
\[f+g :C\Sigma A'\xrightarrow{C\nu} C(\Sigma A')\vee C(\Sigma A')\xrightarrow{\nabla(f\vee g)} X.\]
Then, following (\textit{mutatis mutandis}) Hardie \cite[Section~1]{hardie} and James \cite[Section~10]{james}, we may state:
\begin{proposition}\label{prop.sep.properties}Let $f,g,h:CA\to X$ be maps such that $f_{\vert A}=g_{\vert A}=h_{\vert A}$, and let $k:X\to Y$ be any map. Then:
\begin{enumerate}[label=\textup{(\roman*)}]
\item\label{prop.dif.1} triviality: $d(f,f)=0$;
\item\label{prop.dif.2} triangularity: $d(f,g)+d(g,h)=d(f,h)$;
\item\label{prop.dif.3} antisymmetry: $d(g,f)=-d(f,g)$;
\item\label{prop.dif.4} naturality: $d(k f,k g)=k  d(f,g)$;
\item\label{prop.dif.4.1} co-naturality: $d((C\alpha)^\ast f,(C\alpha)^\ast g)=(\Sigma \alpha)^\ast d(f,g)$ for $\alpha:A'\to A$.
\end{enumerate}

Moreover, let $f_i,g_i:C\Sigma A\to X$ be maps for $i=1,2$ such that $f_{1\vert \Sigma A}=f_{2\vert \Sigma A}$ and  $g_{1\vert \Sigma A}=g_{2\vert \Sigma A}$. Then:
\begin{enumerate}[resume*]
\item\label{prop.dif.5}  $d(f_1+g_1,f_2+g_2)=d(f_1,f_2)+d(g_1,g_2)$;

\item\label{prop.dif.6} if  $\varphi: \Sigma A \to X$ then there exists $f':CA\to X$ such that $f'_{\vert A}=f_{\vert A}$ and $-\varphi=d(f,f')$;
\end{enumerate}

Let $j_\ast$ and $\delta_\ast$ be the maps as in the exact sequence~\eqref{eq.seq.puppe}. Then:
\begin{enumerate}[resume*]
\item\label{prop.dif.7} if $\delta_\ast(f',f'')=0$ for some $(f',f'')\in \pi_1(\Sigma A,k)$ and $g\in \pi(C\Sigma A,X)$ such that
$g_{\vert \Sigma A}=f'_{\vert \Sigma A}$ then $j_*d(f',kg)=(f',f'')$. In addition for any $h\in j_\ast^{-1}(f',f'')$ there exists $g'\in \pi(C\Sigma A,X)$ with $g'_{\vert \Sigma A}=f'_{\vert \Sigma A}$ such that $h=d(f',kg')$.
\end{enumerate}
\end{proposition}

\begin{proof}
Since the items \ref{prop.dif.1}--\ref{prop.dif.5} follow directly from the definition of $d$, we just prove \ref{prop.dif.6} and sketch a proof of  \ref{prop.dif.7}.
\par For \ref{prop.dif.6}, let $q:A\times I \to CA$ be the quotient map such that the inclusion $A\hookrightarrow CA$ is given by $a\mapsto q(a,1)=[a,1]$ for $a\in A$. Define $f':CA\to X$ by
\[f'[a,t]=\begin{cases} \varphi[a,2t], & 0\leq t \leq \frac{1}{2},\\[1ex]
f q (a,2t-1), & \frac{1}{2}\leq t \leq 1,
\end{cases}\]
and note that for $t=\frac{1}{2}$, $\varphi[a,1]=\ast=f q(a,0)$, and for $t=1$, $f'[a,1]=f q(a,1)=f[a,1]$, that is, $f_{\vert A}=f'_{\vert A}$. After some computations and changes of parameters we can show that: \[d(f,f')[a,t]=\begin{cases}
f[a,\frac{8}{3}t], & 0\leq t \leq \frac{3}{8}, \\[1ex]
f[a,2-\frac{8}{3}t], &  \frac{3}{8}\leq t \leq \frac{3}{4}, \\[1ex]
\varphi[a,4(1-t)], & \frac{3}{4}\leq t \leq 1.
\end{cases}\]
Since the first two parts provide a map homotopic to $d(f,f)=0$ and the third part is homotopic to $-\varphi$, we obtain $-\varphi= d(f,f')$.

For \ref{prop.dif.7}, we note that the map $j_*:\pi(\Sigma^2A,Y)\to \pi_1(\Sigma A,k)$ is determined by the map $(C\Sigma A,\Sigma A)\to (\Sigma^2A,\ast)$. Then, we simply mimic the proof of \cite[Theorem~1.9]{hardie}.
\end{proof}

We finish this section by stating a relation between the separation element and the products defined in Section~\ref{sec.RelGWP} (cf.\ \cite[(2.7)]{hardie}).

\begin{proposition}\label{prop.separation} Let $f=(f_1,f_2)$, $f'=(f_1',f_2')\in \pi_1(\Sigma A,k)$ and $g\in \pi(\Sigma B,X)$, where $k:X\to Y$. Suppose that $f_2=f_2':\Sigma A\to X$.
Then \begin{equation*}
d\bigl( [f_1,kg]', [f_1',kg]' \bigr)= [ d(f_1,f_1') ,kg ]:\Sigma^ 2(A\wedge B)\to Y.
\end{equation*}
\end{proposition}

\begin{proof}
Just observe that the restrictions of  $[f_1,kg]'$ and $[f_1',kg]'$ to $\Sigma(A\wedge B)$ are equal to $k_*[f_2,g]$
and $k_*[f_2',g]$, respectively. Since $f_2=f_2'$, we get $k_*[f_2,g]=k_*[f_2',g]$ and $d\bigl( [f_1,kg]', [f_1',kg]' \bigr)$ is defined.
Then by the definition of the generalized separation element and the discussion above $d\bigl( [f_1,kg]', [f_1',kg]' \bigr)=[ d(f_1,f_1') ,kg ]''=
[ d(f_1,f_1') ,kg ]$,
 and the proof follows.
\end{proof}

\section{Higher order generalized Whitehead product}\label{sec.HOGWP}
Let $r\geq 2$ be an integer and denote by $\underline{A}=(A_1,\dots ,A_r)$ an $r$-tuple of topological spaces with base points $\ast$. The \textit{fat wedge} of $\underline{A}$ is the space
\[FW(\underline{A})=\bigl\{(a_1,\dots ,a_r)\in A_1\times \cdots\times A_r \mid \text{at least one }a_i=\ast \bigr\}.\]
If $r=2$, then $FW(\underline{A})=A_1\vee A_2$ is the wedge sum.

Following Porter's notation \cite{porter}, let
$T_0(\underline{A})= A_1\times\cdots\times A_r$ and $T_s(\underline{A}) \subseteq T_0(\underline{A})$ be the subset
of points $(a_1,\dots ,a_r)$ with at least $s$ coordinates $a_i=\ast$ for $s=1,\dots ,r$. Thus, $T_{r-1}(\underline{A})= A_1\vee \cdots
\vee  A_r$ is the wedge sum and $T_1(\underline{A})=FW(\underline{A})$ is the fat wedge. Also, we write $\Lambda(\underline{A})=A_1\wedge\cdots\wedge A_r$  and $\Sigma(\underline{A})=(\Sigma A_1,\dots,\Sigma A_r)$.

\begin{remark}\label{rem.inclusions}In the sequel, we need to take some coordinate $a_{i_0}=\ast$. To do this, we define
\begin{equation*}\label{eq.T0i}
T_0^{(i)}(\underline{A}) =  \textstyle\prod\limits_{\substack{j=1\\ j\neq i}}^{r}  A_j, \quad \Lambda^{(i)}(\underline{A})=\textstyle\bigwedge\limits_{\substack{j=1\\ j\neq i}}^{r} A_j
\end{equation*}
for $i=1,\dots ,r$.
In this way, the subset $T_s^{(i)}(\underline{A}) \subseteq T_0^{(i)}(\underline{A})$ is defined for $s=1,\dots ,r-1$ and has a
self-explanatory notation. Also, there are canonical embeddings
$\Psi_s^{(i)}:T_s^{(i)}(\underline{A}) \hookrightarrow T_{s+1}(\underline{A})$ for any $s=0,\dots ,r-1$.
\end{remark}

In \cite{porter}, Porter constructed \textit{the generalized Whitehead map} (cf.~\eqref{eq.whitehead.map.classical})
\begin{equation}\label{eq.whitehead.map.generalized}
\omega_r : \Sigma^{r-1}\Lambda(\underline{A})\to T_1\Sigma(\underline{A})
\end{equation}
to define the \textit{$r$\textsuperscript{th} order generalized Whitehead product} $\omega_r(f)\in \pi(\Sigma^{r-1}\Lambda(\underline{A}),X)$ of a map $f:T_1\Sigma(\underline{A}) \to X$ as the composite
\begin{equation*}\label{eq.r-WP.}
\omega_r(f): \Sigma^{r-1}\Lambda(\underline{A}) \xrightarrow{\omega_r} T_1\Sigma(\underline{A}) \stackrel{f}{\longrightarrow} X.
\end{equation*}

Given $f_i:\Sigma A_i \to X$ for $i=1,\dots ,r$, let  $f_1\vee\cdots\vee f_r:T_{r-1}\Sigma (\underline{A}) \to X$ be the wedge sum map. The (possibly empty) set
\begin{equation*}
[f_1,\dots ,f_r]=\bigl\{\omega_r(f)\in \pi(\Sigma^{r-1}\Lambda(\underline{A}),X) \bigr\},\end{equation*}  for all
$f:T_1\Sigma(\underline{A}) \to X$ extending $f_1\vee\cdots\vee f_r$,
is called the \textit{$r$\textsuperscript{th} order generalized Whitehead product} of $f_1,\dots ,f_r$.  Hardie \cite{hardie1} gave
the definition of $[f_1,\dots ,f_r]$ when all the $A_i$'s are spheres (called the $r$\textsuperscript{th} order spherical Whitehead product).

In virtue of \cite[Theorem~(2.7)]{porter}, it is non-empty if and only if all the lower products $[f_{m_1},\dots ,f_{m_k}]$ for $1\le m_1\le \cdots \le m_k\le r$ with $k=2,\dots , r-1$ contain the zero element $0$. We note that the set $[f_1,\dots ,f_r]$ can be empty even if $f_i=0$ for some $i$. This is the case for $[0,\iota_2,\iota_2]$, since the classical Whitehead product $[\iota_2,\iota_2]=2\eta_2\neq 0$, by \cite[Chapter V]{toda}.

\begin{remark}\label{rem.permutation}In the sequel we will use the same notation $\sigma$ for a permutation of the set $\{1,\dots,r\}$ and for its induced homeomorphism $A_1\wedge\dots\wedge A_r\to A_{\sigma(1)}\wedge\dots\wedge A_{\sigma(r)}$.
\end{remark}

\begin{proposition}\label{prop.permutation}
{Let $f_i:\Sigma A_i \to X$ for $i=1,\dots ,r$ and $\sigma\in S_r$ be a permutation of the set $\{1,\dots ,r\}$. Then
\[ [f_{\sigma(1)},\dots ,f_{\sigma(r)}]=\operatorname{sgn}(\sigma)(\Sigma^{r-1}\sigma)^*[f_1,\dots ,f_r]. \]}
\end{proposition}

\begin{proof}Denote by $\sigma\underline{A}=(A_{\sigma(1)},\dots ,A_{\sigma(r)})$ the image of $\underline{A}$ under the permutation
$\sigma\in S_r$.
It is clear that there is a commutative diagram
\[
\xymatrix{
\Sigma^{r-1}\Lambda(\underline{A})  \ar[r]^{\omega_r} & T_1\Sigma(\underline{A})  \\
\Sigma^{r-1}\Lambda(\sigma\underline{A}) \ar[u]^{\Sigma^{r-1}\sigma} \ar[r]^{\sigma\omega_r} &
T_1\Sigma(\sigma\underline{A}) \ar[u]_{T_1\Sigma\sigma} \rlap{.}
}
\]
Let $f_{\sigma(1)}\vee \cdots\vee f_{\sigma(r)}=(T_{r-1}\Sigma\sigma)^*(f_1\vee\cdots\vee f_r): T_{r-1}\Sigma(\sigma\underline{A})\to X$. So, if $f:T_1\Sigma(\underline{A})\to X$
is an extension of $f_{1}\vee \cdots\vee f_{r}$ then $\sigma f=(T_1\Sigma\sigma)^*(f):T_1\Sigma(\sigma\underline{A})\to X$ is an extension
of $f_{\sigma(1)}\vee \cdots\vee f_{\sigma(r)}$. Finally, note that $\sigma\omega_r(\sigma f)=\operatorname{sgn}(\sigma)(\Sigma^{r-1}\sigma)^*(\omega_r(f))$ and the proof follows.
\end{proof}

According to \cite[Theorem~(2.3)]{porter} there is a homotopy equivalence
\begin{equation}T_1\Sigma(\underline{A})\cup_{\omega_r} C(\Sigma^{r-1}\Lambda(\underline{A}))\xrightarrow{\simeq} T_0\Sigma(\underline{A}).
\label{equ}\end{equation}
So, we have a pair-map $(\Omega_r,\omega_r)$ commuting the diagram
\[
\xymatrix{
\Sigma^{r-1}\Lambda(\underline{A}) \ar@{^{(}->}[d] \ar[r]^-{\omega_r} & T_1\Sigma(\underline{A})\ar@{^{(}->}[d] \\
C(\Sigma^{r-1}\Lambda(\underline{A})) \ar[r]^-{\Omega_r} & T_0\Sigma(\underline{A})\rlap{.}
}
\]

Considering the spaces and inclusions as in Remark~\ref{rem.inclusions}, let \[\theta: \textstyle\bigvee\limits_{i=1}^r \Sigma^{r-2}\Lambda^{(i)}
(\underline{A}) \to T_2\Sigma(\underline{A})\] be given by $\theta_{\vert\Sigma^{r-2}\Lambda^{(i)}(\underline{A})}=\Psi_1^{(i)} \omega_r^{(i)}$, where $\omega_r^{(i)}:\Sigma^{r-2}\Lambda^{(i)}(\underline{A})\to T_1^{(i)}\Sigma(\underline{A})$ is the generalized Whitehead map. Then, as a particular case of \cite[Theorem~(2.6)]{porter}, we have:
\begin{theorem}\label{teo.porter.2.6}{The following spaces are homotopy equivalent:
\begin{enumerate}[label=\textup{(\roman*)}]
\item $T_1\Sigma(\underline{A})$;
\item $T_2\Sigma(\underline{A})\cup_\theta C\left( \textstyle\bigvee\limits_{i=1}^r \Sigma^{r-2}\Lambda^{(i)}(\underline{A})\right)$;
\item $T_2\Sigma(\underline{A})\cup_\theta \textstyle\bigvee\limits_{i=1}^r C(  \Sigma^{r-2}\Lambda^{(i)}(\underline{A}))$.
\end{enumerate}}
\end{theorem}

Now, we discuss some results on the $r$\textsuperscript{th} order generalized Whitehead product $[f_1,\dots ,f_r]$ of maps $f_i:\Sigma A_i\to X$
for $i=1,\dots ,r$. First, we recall \cite[Theorems~(2.1)(d), (2.4), (2.5) and Remark~(4)]{porter} and present a corollary of \cite[Theorem~(2.3)]{porter}.

\begin{theorem}
\begin{enumerate}[label=\textup{(\roman*)}]\thmenumhspace
\item\label{teo.porter.2.1} \textup{(\textit{Naturality})} Let $f_i:\Sigma A_i\to \Sigma B_i$, $g_i:\Sigma B_i\to X$ and $k:X\to Y$ be maps for $i=1,\dots , r$. Then,
$k_*[g_1,\dots ,g_r]\subseteq [k g_1,\dots ,k g_r]$.
\item\label{teo.porter.2.5} If $X$ is an $H$-space then $\omega_r(f)=0$.
\item\label{lemma.zerobelongs} Let $f_i:\Sigma A_i\to X$ be maps for $i=1,\dots ,r$. Then $0\in [f_1,\dots ,f_r]$ if and only if
$f_1\vee\cdots\vee f_r:T_{r-1}\Sigma(\underline{A})\to X$ has an extension to $T_0\Sigma(\underline{A})$.
Further, $\Sigma\omega_r(f)=0$ for $f:T_1\Sigma(\underline{A}) \to X$ extending $f_1\vee\dots\vee f_r$.
\end{enumerate}\label{teo.porter.glue}
\end{theorem}

\begin{corollary}\label{cor.fi.zero}
If $[f_1,\dots ,f_r]\neq\emptyset$ and $f_{i_0}=0_{i_0}$ for some $1\leq i_0\leq r$ then $0\in [f_1,\dots ,f_r]$.
\end{corollary}
\begin{proof}In virtue of Proposition~\ref{prop.permutation}, we can suppose that $i_0=1$. By hypothesis the map
$0\vee f_2\vee\cdots\vee f_r: T_{r-1}\Sigma (\underline{A}) \to X$ has an extension $F_1:T_1\Sigma(\underline{A})\to X$.
Define $F: T_0\Sigma(\underline{A}) \to X$ by $F(x_1,\dots ,x_r)= F_1(\ast,x_2,\dots ,x_r)$ for $(x_1,\dots,x_r)\in T_0\Sigma(\underline{A})$.
Then, $F$ is an extension of $0_1\vee f_2\vee\cdots\vee f_r$, and the result follows by Theorem~\ref{teo.porter.glue}\ref{lemma.zerobelongs}.
\end{proof}

Recall from \cite{Si} that a space $X$ is called a  \textit{$G$-space} if for any map $f : \mathbb{S}^n\to X$ and $n\ge 1$ the map
$\nabla (\iota_X \vee f): X\vee \mathbb{S}^n \to X$ extends to $F: X \times \mathbb{S}^n \to X$.
According to \cite[Theorem~3.2]{Si}, any $r$\textsuperscript{th} order spherical Whitehead product of a  $G$-space contains zero.

Now, let $\mathcal{G}(A, X)\subseteq \pi(A,X)$ be the set of (homotopy classes of) maps $f:A \to X$ such that the maps $\nabla (\iota_X \vee f): X\vee A \to X$
extend to $F: X \times A \to X$. A space $X$ is called a \textit{$G$-space with respect to the space $A$} if $\mathcal{G}(A, X)=\pi(A,X)$.

\begin{lemma}\label{lemma.Gspace}
{If $X$ is a $G$-space with respect to the spaces $ A_1,\dots , A_r$ then any map $\iota_X\vee f_1\vee\cdots\vee f_r : X\vee A_1\vee\cdots\vee A_r\to X$
extends to $F: X\times A_1\times\cdots\times A_r\to X$.}
\end{lemma}

\begin{proof} The proof is by induction on $r$. Let $r=2$ and  $F_i:X\times  A_i\to X$ extending $\iota_X\vee f_i:X\vee  A_i\to X$ for $i=1,2$. Then $F_2(F_1\times \iota_{ A_2}):X\times  A_1\times  A_2\to X$ extends $\iota_X\vee f_1\vee f_2: X\vee  A_1\vee  A_2 \to X$.

Suppose that the statement is true for ${r-1}$, that is, there exists $F:X\times  A_1\times \cdots \times  A_{r-1}\to X$ extending $\iota_X\vee f_1\vee\cdots\vee f_{r-1}$.

Finally, let $F_r:X\times  A_r\to X$ extending $\iota_X\vee f_r:X\vee  A_r\to X$ and then the composite $F_r (F\times \iota_{ A_r}): X\times T_0(\underline{A})\to X$ extends $\iota_X\vee f_1\vee\cdots\vee f_r$, which completes the proof.
\end{proof}

Referring to Williams~\cite{williams}, we say that a space $X$ has {property $P_r$} if for every $f_i : \Sigma A_i\to X$ with $i=1,\dots ,r$, we have
$0\in[f_1,\dots ,f_r]$. Certainly, in view of Theorem~\ref{teo.porter.glue}\ref{teo.porter.2.5}, any $H$-space has not only property $P_r$ for all
$r\ge 2$ but $0$ is the only element of $[f_1,\dots ,f_r]$. (Williams \cite{williams}: \textit{We note at this point that it is unresolved conjecture as to whether $X$ has property $P_r$ implies that $0$ is the {only} element of $[f_1,\dots ,f_r]$.})

Directly from Theorem~\ref{teo.porter.glue}\ref{lemma.zerobelongs} and Lemma~\ref{lemma.Gspace} it follows:

\begin{corollary}
{ Every  $G$-space $X$ with respect to any $\Sigma A_1,\dots ,\Sigma A_r$ has pro\-perty $P_r$. In particular, If $f : \Sigma A\to X$ and $X$ is
a  $G$-space with respect to any $\Sigma A$ then $0\in [f,\overset{\times r}{\dotsc{}},f]$ for any $r\geq 2$.}
\end{corollary}

In general, the generalized Whitehead product is not additive. But, in the sequel, we need  an addition operation defined in \cite{porter} for some particular functions.
We say that maps $f,g:T_1\Sigma(\underline{A})\to X$ are \textit{compatible off the $i$\textsuperscript{th} coordinate} if $f\Psi_0^{(i)}=
g\Psi_0^{(i)}:T_0^{(i)}\Sigma(\underline{A})\to X$. If $A_i=\Sigma A_i'$ is a suspension one defines an addition
$+^{(i)}$ by
\begin{multline*}
(f+^{(i)}g)(x_1,\dots ,[t_i,[u,a_i']],\dots ,x_r)\\
=\begin{cases}
f(x_1,\dots ,[t_i,[2u,a_i']],\dots ,x_r), & 0\leq u \leq \frac{1}{2}, \\[1ex]
g(x_1,\dots ,[t_i,[2u-1,a_i']],\dots ,x_r), & \frac{1}{2}\leq u \leq 1.
\end{cases}
\end{multline*}
Also, one defines ${-^{(i)}}f$ by \[(-^{(i)}f)(x_1,\dots ,[t_i,[u,a_i']],\dots ,x_r)=f(x_1,\dots ,[t_i,[1-u,a_i']],\dots ,x_r).\]

Then, in view of \cite[Theorem (2.13)]{porter}, we can state:

\begin{theorem}
\begin{enumerate}[label=\textup{(\roman*)}]\thmenumhspace
\item $\omega_r(f+^{(i)}g)=\omega_r(f)+\omega_r(g)$;

\item $\omega_r({-^{(i)}}f)=-\omega_r(f)$.
\end{enumerate}
\end{theorem}
From this, it follows:

\begin{corollary}\label{c1} { If $f_i : \Sigma A_i\to X$ for $i=1,\dots ,r$ then \[[f_1,\dots ,nf_i,\dots ,f_r]\subseteq n[f_1,\dots ,f_i,\dots ,f_r]\]
for any integer $n$ and $i=1,\dots ,r$.}
\end{corollary}

\section{Triple spherical Whitehead products}\label{sec.triple}
E.\ C.\ Zeeman \cite[Section~3.3]{massey} associated with a triple of elements $f\in \pi_p(X)$, $g\in \pi_q(X)$ and
$h\in \pi_r(X)$ ($p,q,r\geq 2$), whose Whitehead products taken in pairs all vanish, a certain subset $[f,g,h]'$
of $\pi_{p+q+r-1}(X)$. K.\ A.\ Hardie \cite{hardie} sharpened this construction defining the triple spherical Whitehead product $[f,g,h]\subseteq [f,g,h]'$
and proving in \cite[Theorem~0.1]{hardie} that $[f,g,h]$ is a coset of the subgroup \[J(f,g,h)=[\pi_{q+r}(X),f]+
[\pi_{p+r}(X),g]+[\pi_{p+q}(X),h]\] of $\pi_{p+q+r-1}(X)$.
Thus, a \textit{trivial} triple spherical Whitehead product means $[f,g,h]=J(f,g,h)$ or equivalently, $0\in[f,g,h]$.
Hardie stated in \cite[Section~5]{hardie}: \textit{We do not know of a case when the triple \textup{(\textit{spherical})} product $[f,g,h]$ is non-trivial for a sphere}, and
for $X=\mathbb{S}^4$, \textit{the triple product $[\eta_4,\eta_4^2,2\iota_4]\subseteq\pi_{14}(\mathbb{S}^4)$ is possibly non-trivial.}
Proposition~\ref{prop.order.15} below shows that it is non-empty and has order fifteen. In order to prove this,
we need to check first that all lower products vanish. We keep below the standard notations from Toda's book \cite{toda}.

\begin{lemma}\label{lemma.foo} { Given $\eta_4\in \pi_5(\mathbb{S}^4)$, $\eta_4^2\in \pi_6(\mathbb{S}^4)$, and $\iota_4\in\pi_4(\mathbb{S}^4)$, the following classical
Whitehead products vanish: \begin{align*} [\eta_4,2\iota_4]&=0, & [\eta_4^2,2\iota_4]&=0, & [\eta_4,\eta_4^2]&=0. \end{align*}}
\end{lemma}
\begin{proof}Certainly, $[\eta_4,2\iota_4]=[2\eta_4,\iota_4]=0$
and $[\eta_4^2,2\iota_4]=[2\eta_4^2,\iota_4]=0$. Further, by \cite[(5.9)]{toda}, $\eta_3\circ \nu_4=\nu'\circ\eta_6$, by \cite[(5.5)]{toda},
$4\nu_n=\eta_n^3$ for $n\geq 5$, and by \cite[Proposition~5.6]{toda}, $\Sigma\nu'$ has order four.

But
$[\eta_4,\eta_4^2] = \eta_4 \circ [\iota_5,\eta_5] = \eta_4 \circ [\iota_5,\iota_5] \circ \eta_9$ and $[\iota_5,\iota_5]=\nu_5\circ\eta_8$, by  \cite[(5.10)]{toda}. Thus
$[\eta_4,\eta_4^2]= (\eta_4\circ \nu_5) \circ \eta_8 \circ\eta_9 = \Sigma\nu' \circ\eta_7 \circ \eta_8 \circ \eta_9 = \Sigma\nu' \circ 4 \nu_7 = 4\Sigma\nu'\circ\eta_7=0$.
\end{proof}

To prove the next proposition, we recall from \cite[Chapter~XIV]{toda} that \begin{equation*}\pi_{11}(\mathbb{S}^4)=\mathbb{Z}_3\{\alpha_2(4)\}\oplus \mathbb{Z}_5\{\alpha_1'(4)\} \end{equation*} and for the $2$-primary components,
in view of \cite[Propositions~5.9 and 5.11]{toda}, we have:
\begin{align*}
\pi_{9}^4&=\mathbb{Z}_2\{\nu_4\circ\eta_7^2\}\oplus\mathbb{Z}_2\{\Sigma\nu'\circ\eta_7^2\}, &
\pi_{10}^4&=\mathbb{Z}_8\{\nu_4^2\}.
\end{align*}

Moreover, from \cite[Proposition~IV.5]{serre},  Serre's isomorphism for any odd $p$-primary components
\[ \pi_{i-1}(\mathbb{S}^{2m-1};p)\oplus\pi_i(\mathbb{S}^{4m-1};p)\simeq \pi_i(\mathbb{S}^{2m};p)\] is given by
$(f,g)\mapsto \Sigma f+[\iota_{2m},\iota_{2m}]\circ g$.

\begin{proposition}\label{prop.order.15} The groups $[\pi_{9}(\mathbb{S}^4),\eta_4^2]$ and $[\pi_{10}(\mathbb{S}^4),\eta_4]$ are trivial.
In particular, $J(\eta_4,\eta_4^2,2\iota_4)=[\pi_{11}(\mathbb{S}^4),2\iota_4]$ and it is a subgroup of $\pi_{14}(\mathbb{S}^4)$ with order fifteen. In addition,
the triple spherical Whitehead product
\[[\eta_4,\eta_4^2,2\iota_4]=(4x(\nu_4\circ\sigma')+2y(\Sigma\varepsilon'))+J(\eta_4,\eta_4^2,2\iota_4)\] for some $x,y\in\{0,1\}$.
\end{proposition}

\begin{proof}
First, recall that for $f\in \pi_k(\mathbb{S}^m)$ and $g\in\pi_l(\mathbb{S}^m)$ with relatively prime orders, the Whitehead product
$[f,g]=0$. Hence, because the orders of $\eta_4$ and $\eta_4^2$ are two, we can restrict to $\pi_{9}^4$ and $\pi_{10}^4$ only.

For the first statement notice that  \[[\nu_4^2,\eta_4]=[\nu_4,\iota_4]\circ \Sigma(\nu_6\wedge\eta_3) \] and
\[[\nu_4\circ\eta_7^2,\eta_4^2]=[\nu_4,\iota_4]\circ \Sigma(\eta_6^2\wedge\eta_3^2). \]
Because $[\nu_4,\iota_4]=\pm 2\nu_4^2$ (\cite[Corollary~(7.4)]{barcus}), we derive that $[\nu_4^2,\eta_4]=0$ and  $[\nu_4\circ\eta_7^2,\eta_4^2]=0$.

Next, $[\Sigma\nu',\iota_4]=[\iota_4,\iota_4]\circ \Sigma\nu'$ and \cite[(5.5), (5.8) and Proposition 5.11]{toda} implies $[\Sigma\nu',\iota_4]=\pm 4\nu_2^2$. Consequently, \[ [\Sigma\nu'\circ\eta_7^2,\iota_4\circ\eta_4^2]=[\Sigma\nu',\iota_4]\circ \Sigma(\eta_6^2\wedge \eta_3^2)=0. \]

This proves that $J(\eta_4,\eta_4^2,2\iota_4)=[\pi_{11}(\mathbb{S}^4),2\iota_4]$.

To determine its order, notice that by means of Serre's isomorphisms \[ \pi_{13}(\mathbb{S}^3;p)\oplus \pi_{14}(\mathbb{S}^7;p)\simeq \pi_{14}
(\mathbb{S}^4;p) \] for $p=3,5$, the orders of $[\alpha_2(4),\iota_4]=[\iota_4,\iota_4]\circ \alpha_2(7)$ and $[\alpha'_1(4),\iota_4]=
[\iota_4,\iota_4]\circ \alpha'_1(7)$ are three and five, respectively. This implies that the order of $J(\eta_4,\eta_4^2,2\iota_4)$ is fifteen.

Finally, let $\alpha\in \pi_{14}(\mathbb{S}^4)$ be such that $[\eta_4,\eta_4^2,2\iota_4]=\alpha+J(\eta_4,\eta_4^2,2\iota_4)$. Since $2\eta_4=0$,
in view of Corollaries~\ref{cor.fi.zero} and \ref{c1}, we deduce that  \[0\in [2\eta_4,\eta_4^2,2\iota_4]\subseteq 2[\eta_4,\eta_4^2,2\iota_4]= 2\alpha +J(\eta_4,
\eta_4^2,2\iota_4).\] Consequently, $2\alpha\in J(\eta_4,\eta_4^2,2\iota_4)$ and so $30\alpha=0$. Hence, $\alpha=a\alpha'+b\beta$ for $\alpha'\in \pi_{14}^4$
with order two and $\beta\in J(\eta_4,\eta_4^2,2\iota_4)$ for some $a\in\{0,1\}$ and $b\in\{0,\dots,14\}$. This implies that
$[\eta_4,\eta_4^2,2\iota_4]=\alpha+J(\eta_4,\eta_4^2,2\iota_4)=a\alpha'+J(\eta_4,\eta_4^2,2\iota_4)$ and $a\alpha'\in [\eta_4,\eta_4^2,2\iota_4]$.

In view of \cite[Theorem 7.3]{toda}, it holds $\pi_{14}^4=\mathbb{Z}_8\{\nu_4\circ\sigma'\}\oplus\mathbb{Z}_4\{\Sigma\varepsilon'\}\oplus\mathbb{Z}_2\{\eta_4\circ\mu_5\}$.
Then, we derive that $\alpha'=4x'(\nu_4\circ\sigma')+2y'(\Sigma\varepsilon')+z'(\eta_4\circ\mu_5)$ for some $x',y',z'\in\{0,1\}$. Because $\alpha'\in [\eta_4,\eta_4^2,2\iota_4]$, we get
from Theorem~\ref{teo.porter.glue}\ref{lemma.zerobelongs} that $0=\Sigma\alpha'=4x'(\nu_5\circ\Sigma\sigma')+2y'(\Sigma^2\varepsilon')+z'(\eta_5\circ\mu_6)$. But, by \cite[(7.10) and (7.16)]{toda}, it holds $2(\nu_5\circ\sigma_8)=\pm\Sigma^2\varepsilon'$
and $2(\nu_5\circ\sigma_8)=\nu_5\circ\Sigma\sigma'$. Thus, $z'(\eta_5\circ\mu_6)=0$, so $z'=0$ and
$\alpha'=4x'(\nu_4\circ\sigma')+2y'(\Sigma\varepsilon')$ for some $x',y'\in\{0,1\}$ which completes the proof.
\end{proof}

\section{Main result}\label{sec.main}
Theorem~\ref{teo.porter.2.6} provides the following commutative diagram
\begin{equation*}
\begin{aligned}
\xymatrix@C=1.5cm@R=.6cm{
\Sigma^{r-2}\Lambda^{(i)}(\underline{A}) \ar@{^{(}->}[d] \ar@{^{(}->}[r] & \textstyle\bigvee\limits_{j=1}^r \Sigma^{r-2}
\Lambda^{(j)}(\underline{A}) \ar@{^{(}->}[d] \ar[r]^(.6){\theta} & T_2\Sigma(\underline{A}) \ar@{^{(}->}[d]^{k} \\
C(\Sigma^{r-2}\Lambda^{(i)}(\underline{A})) \ar@{^{(}->}[r] & C\left(\textstyle\bigvee\limits_{j=1}^r \Sigma^{r-2}\Lambda^{(j)}
(\underline{A})\right) \ar[r] & T_1\Sigma(\underline{A})
}
\end{aligned}
\end{equation*}
which lead to pair-maps
\begin{equation}\label{eq.Theta}
(h_i',h_i''): \bigl(C(\Sigma^{r-2}\Lambda^{(i)}(\underline{A})), \Sigma^{r-2}\Lambda^{(i)}(\underline{A}) \bigr) \to (T_1\Sigma(\underline{A}), T_2\Sigma(\underline{A}) ),
\end{equation}
for $i=1,\dots ,r$.

Let $\Theta_i=(h_i',h_i'')\in \pi_1(\Sigma^{r-2}\Lambda^{(i)}(\underline{A}),k)$ and
$\jmath_i:\Sigma A_i \hookrightarrow T_2\Sigma(\underline{A})$ be the canonical inclusion map for $i=1,\dots ,r$, where
$k:T_2\Sigma(\underline{A})\hookrightarrow T_1\Sigma(\underline{A})$ is the inclusion map.
Then each relative generalized Whitehead product
\begin{multline*}
[\Theta_i,\jmath_i]_R : \bigl( C(\Sigma(\Sigma^{r-3}\Lambda^{(i)}(\underline{A})\wedge A_i)),  \Sigma(\Sigma^{r-3}\Lambda^{(i)}(\underline{A}) \wedge A_i) \bigr)  \\   \to (T_1\Sigma(\underline{A}),T_2\Sigma(\underline{A}))
\end{multline*}
simplifies to
\begin{equation*}
\tau_i^*[\Theta_i,\jmath_i]_R : \bigl(C(\Sigma^{r-2}\Lambda(\underline{A}) ) , \Sigma^{r-2}\Lambda(\underline{A})\bigr) \to (T_1\Sigma(\underline{A}), T_2\Sigma(\underline{A}) ),
\end{equation*}
where $\tau_i=(C\Sigma^{r-2}\sigma_i,\Sigma^{r-2}\sigma_i)$ for $i=1,\dots,r$, and $\sigma_i\in S_r$ is the permutation inducing the homeomorphism $\sigma_i: \Lambda(\underline{A})\to \Lambda^{(i)}(\underline{A})\wedge A_i$.

Hence, we got the relative generalized Whitehead products $\tau_i^*[\Theta_i,\jmath_i]_R\in \pi_1(\Sigma^{r-2}\Lambda(\underline{A}), k)$
for $i=1,\dots ,r$.

Next, consider the commutative diagram
\begin{equation*}
\xymatrix@R=1cm@C=.8cm{ \pi_1(\Sigma^{r-1}\Lambda(\underline{A}),k') \ar[d]_-{\delta_{\ast 1}} \ar@/^4ex/[dr]^-{\partial} & \\
\pi(\Sigma^{r-1}\Lambda(\underline{A}),T_1\Sigma(\underline{A})) \ar[r]^-{j_\ast} &
\pi_1(\Sigma^{r-2}\Lambda(\underline{A}),k) \ar[r]^-{\delta_{\ast 2}}
& \pi(\Sigma^{r-2}\Lambda(\underline{A}),T_2\Sigma(\underline{A}))\rlap{,}}
\end{equation*}
where $k:T_2\Sigma(\underline{A})\hookrightarrow T_1\Sigma(\underline{A})$ and $k':T_1\Sigma(\underline{A})\hookrightarrow T_0\Sigma(\underline{A})$ are inclusion maps, $\delta_{\ast 1},\delta_{\ast 2}$ are boundaries and $j_*$ is the obvious map.

Since the bottom row is exact and $\delta_{\ast 1}(\Omega_r,\omega_r)=\omega_r$, we obtain $\partial(\Omega_r,\omega_r)=j_\ast(\omega_r)$ and so
$\delta_{\ast 2}\partial(\Omega_r,\omega_r)=\delta_{\ast 2} j_*(\omega_r)=0$.

Following {\em mutatis mutandis} the result due to Nakaoka--Toda \cite[Lemma (1.2)]{nakaoka-toda} for spheres and
a proof of its generalization \cite[Formula (0.1)]{hardie2}, we may state for suspensions:
\begin{lemma}\label{lemma.h}If
$ h=\sum\limits_{i=1}^{r} \tau_i^* [\Theta_i,\jmath_i]_R \in \pi_1(\Sigma^{r-2}\Lambda(\underline{A}),k)$ then $h=\partial(\Omega_r,\omega_r)$.
\end{lemma}

Notice that in view of the formula \eqref{eq.trick}, we have \[ [\Theta_i,\jmath_i]_R=\bigl([h'_i,k\jmath_i]',[h''_i,\jmath_i]\bigr)\]
for $i=1,\dots,r$. Consequently,
\[\textstyle h= (h',h'') =\left(\sum\limits_{i=1}^{r}(C\Sigma^{r-2}\sigma_i)^\ast[h'_i,k\jmath_i]',\sum\limits_{i=1}^{r}(\Sigma^{r-2}\sigma_i)^\ast[h''_i,\jmath_i] \right).\]

Now, we are in a position to generalize the main result of \cite[Theorem~4.3]{hardie}.

\begin{theorem}\label{teo.main}Let $f:T_2\Sigma(\underline{A})\to X$ be any map, $\jmath_i:\Sigma A_i \hookrightarrow T_2\Sigma(\underline{A})$
be the canonical inclusion maps, and $f_i=f \jmath_i:\Sigma A_i\to X$ for $i=1,\dots ,r$.
\begin{enumerate}[label=\textup{(\roman*)}]
\item\label{teo.main.a} If $f',f'':T_1\Sigma(\underline{A}) \to X$ are  extensions of $f$ then \begin{equation*}
\omega_r({f'})-\omega_r({f''}) \in \textstyle\sum\limits_{i=1}^{r} (\Sigma^{r-1}\sigma_i)^*[ \pi(\Sigma^{r-1}\Lambda^{(i)}(\underline{A}),X ),f_i ],
\end{equation*}
where $\sigma_i\in S_r$ is the permutation inducing the homeomorphism\linebreak $\sigma_i:  \Lambda(\underline{A}) \to \Lambda^{(i)}(\underline{A})\wedge A_i$ for $i=1,\dots ,r$.

\item\label{teo.main.b} If $f': T_1\Sigma(\underline{A})\to X$ is an extension of $f$ then for any  \[\gamma \in \textstyle\sum\limits_{i=1}^{r} (\Sigma^{r-1}\sigma_i)^*[ \pi(\Sigma^{r-1}\Lambda^{(i)}(\underline{A}),X ),f_i ], \] there exists $f'': T_1\Sigma(\underline{A})\to X$ an extension of $f$ such that $\gamma=\omega_r({f'})-\omega_r({f''})$.
\end{enumerate}
\end{theorem}

\begin{proof}\ref{teo.main.a} By Lemma~\ref{lemma.h} and commutativity of the diagram above, $h=\partial(\Omega_r,\omega_r)=j_*(\omega_r)$ so
that $\omega_r\in j_*^{-1}(h)$. In view of Proposition~\ref{prop.sep.properties}\ref{prop.dif.7}, there exists $g':C\Sigma^{r-2}\Lambda(\underline{A})\to T_2\Sigma(\underline{A})$
with $g'_{\vert \Sigma^{r-2}\Lambda(\underline{A})}=h'_{\vert \Sigma^{r-2}\Lambda(\underline{A})}=h''$ such that $\omega_r=d(h',kg')\in \pi(\Sigma^{r-1}\Lambda(\underline{A}),T_1\Sigma(\underline{A}))$.

Since $f'k=f''k=f$, by \eqref{eq.double} and Propositions~\ref{prop.sep.properties}--\ref{prop.separation}, we get
\begin{align*}
\omega_r(f')-\omega_r(f'') &= d(f'  h',f' k g')- d(f''  h',f'' k g')   \\
&= d(f'  h',f''  h')  \\
&= d\left(\textstyle\sum\limits_{i=1}^{r} f'\bigl((C\Sigma^{r-2}\sigma_i)^\ast[h'_i,k\jmath_i]'\bigr), \textstyle\sum\limits_{i=1}^{r} f''\bigl((C\Sigma^{r-2}\sigma_i)^\ast[h'_i,k\jmath_i]'\bigr)\right)\\
&= \textstyle\sum\limits_{i=1}^{r} d\bigl((C\Sigma^{r-2}\sigma_i)^\ast[f'h'_i,f'k \jmath_i]',(C\Sigma^{r-2}\sigma_i)^\ast[f''h'_i,f''k\jmath_i]'\bigr) \\
&= \textstyle\sum\limits_{i=1}^{r} (\Sigma^{r-1}\sigma_i)^* [d(f'h'_i,f''h'_i),f_i]\\
&\in \textstyle\sum\limits_{i=1}^{r} (\Sigma^{r-1}\sigma_i)^*[ \pi(\Sigma^{r-1}\Lambda^{(i)}(\underline{A}),X ),f_i ].
\end{align*}

\ref{teo.main.b} Let  $\gamma=\textstyle\sum\limits_{i=1}^{r}(\Sigma^{r-1}\sigma_i)^*[\gamma_i,f_i]$, with
$\gamma_i: \Sigma^{r-1}\Lambda^{(i)}(\underline{A})\to X$.
By~\eqref{eq.Theta}, we obtain $f'  h'_i:C(\Sigma^{r-2}\Lambda^{(i)}(\underline{A}))\to X$ and then,
in view of Proposition~\ref{prop.sep.properties}\ref{prop.dif.6},
there exist maps $g_i:C(\Sigma^{r-2}\Lambda^{(i)}(\underline{A}))\to X$ such that the restrictions\linebreak
${g}_{i\vert \Sigma^{r-2}\Lambda^{(i)}(\underline{A})}=f h''_{i\vert \Sigma^{r-2}\Lambda^{(i)}(\underline{A})}$ and $d(f'  h'_i,g_i)=\gamma_i$ for $i=1,\dots ,r$.

Note that the maps $g_i$ determine a map \[g=g_1\vee \cdots \vee g_r:C\left(\textstyle\bigvee\limits_{j=1}^r \Sigma^{r-2}\Lambda^{(j)}(\underline{A})\right) \to X.\]
Since, by Theorem~\ref{teo.porter.2.6}, the space $T_1\Sigma(\underline{A})$ is a push-out, the universal property guarantees the existence of $f'':T_1\Sigma(\underline{A})\to X$ with $f''_{\vert T_2\Sigma(\underline{A})}=f'_{\vert T_2\Sigma(\underline{A})}=f$ and $f''  h'_i=g_i$.
\begin{equation*}
\begin{aligned}
\xymatrix@C=1cm{
\Sigma^{r-2}\Lambda^{(i)}(\underline{A}) \ar@{^{(}->}[d] \ar@{^{(}->}[r] & \textstyle\bigvee\limits_{j=1}^r \Sigma^{r-2}
\Lambda^{(j)}(\underline{A}) \ar@{^{(}->}[d] \ar[r]^-{\theta} & T_2\Sigma(\underline{A}) \ar@{^{(}->}[d]^{k} \ar@/^3ex/[rdd]^{f} \\
C(\Sigma^{r-2}\Lambda^{(i)}(\underline{A})) \ar@{^{(}->}[r] \ar@/_4ex/[rrrd]_-{g_i}  & C\left(\textstyle\bigvee\limits_{j=1}^r
\Sigma^{r-2}\Lambda^{(j)}(\underline{A})\right) \ar[r] \ar@/_3ex/[rrd]^-{g} & T_1\Sigma(\underline{A}) \ar@{-->}[rd]^{f''} \\
&&& X}
\end{aligned}
\end{equation*}

The same computations as in~\ref{teo.main.a} show that $\gamma=\omega_r(f')-\omega_r(f'')$ and the proof is complete.
\end{proof}

\begin{corollary}\label{CC}
\begin{enumerate}[label=\textup{(\roman*)}]\thmenumhspace
\item\label{CC.1} Let $f_i:\Sigma A_i\to X$ be any maps for $i=1,\dots,r$. Suppose that $[f_1,\dots ,f_r]\neq \emptyset$. Then for any $\beta\in [f_1,\dots ,f_r]$ there is an inclusion
\[\beta + \textstyle\sum\limits_{i=1}^{r} (\Sigma^{r-1}\sigma_i)^*[ \pi(\Sigma^{r-1}\Lambda^{(i)}(\underline{A}),X ),f_i ] \subseteq [f_1,\dots ,f_r]. \]

\item\label{CC.2} If $\omega_r(f)\in [f_1,\dots ,f_r]$ for $f: T_1\Sigma(\underline{A})\to X$ then
\begin{multline*}
\textstyle\sum\limits_{i=1}^{r} (\Sigma^{r-1}\sigma_i)^*[ \pi(\Sigma^{r-1}\Lambda^{(i)}(\underline{A}),X ),f_i ]=\\
\bigl\{ \omega_r(f')-\omega_r(f);\ \text{$f': T_1\Sigma(\underline{A})\to X$ with $f'_{\vert T_2\Sigma(\underline{A})}={f}_{\vert T_2\Sigma(\underline{A})}$} \bigr\}.
\end{multline*}
\item\label{CC.3}
If $[f_1,f_2,f_3]\neq\emptyset$ then it is a coset of the subgroup
\begin{multline*}
(\Sigma^2\sigma_1)^*[\pi(\Sigma^2\Lambda^{(1)}(\underline{A}),X),f_1]+
(\Sigma^2\sigma_2)^*[\pi(\Sigma^2\Lambda^{(2)}(\underline{A}),X),f_2]+\\
(\Sigma^2\sigma_3)^*[ \pi(\Sigma^2\Lambda^{(3)}(\underline{A}),X ),f_3 ]\subseteq \pi(\Sigma^2\Lambda(\underline{A}),X).
\end{multline*}
\end{enumerate}
\end{corollary}

\begin{proof}\ref{CC.1} First, for simplicity, let $\Gamma=\textstyle\sum\limits_{i=1}^{r} (\Sigma^{r-1}\sigma_i)^*[ \pi(\Sigma^{r-1}\Lambda^{(i)}(\underline{A}),X ),f_i ]$.
Next, let $f':T_1\Sigma(\underline{A})\to X$ be a map such that $\beta=\omega_r(f')$. Let $f=f'_{\vert T_2\Sigma(\underline{A})}:T_2\Sigma(\underline{A})\to X$ so that $f'$ extends $f$, obviously. By Theorem~\ref{teo.main}\ref{teo.main.b}, for any $\gamma\in \Gamma$ there exists $f'': T_1\Sigma(\underline{A})\to X$ an extension of $f$ such that $\gamma=\omega_r({f'})-\omega_r({f''})$,
or equivalently, $\beta+(-\gamma)=\omega_r(f'')$. Finally, since $f''$ extends the wedge sum map $f_1\vee \cdots\vee f_r$, then $\omega_r(f'')\in [f_1,\dots ,f_r]$ and the result follows.

\ref{CC.2} This follows directly from Theorem~\ref{teo.main}.

\ref{CC.3} This is a consequence of \ref{CC.1} and \ref{CC.2}.
\end{proof}

\section{Some computations for spheres and projective spaces}\label{sec.five}

Because $0\in [0_1,0_2,0_3]$, Corollary~\ref{CC}\ref{CC.3} implies $[0_1,0_2,0_3]=0$ (modulo indeterminacy). Further, we can easily show:

\begin{lemma}\label{Ll}Let $0_i : \Sigma A_i\to X$ be trivial maps for $i=1,\dots,r$ with $r\ge 3$. The following are equivalent:
\begin{enumerate}[label=\textup{(\roman*)}]
\item\label{Ll.i} $[0_1,\dots,0_r]=0$ \textup{(\textit{modulo indeterminacy})} for any space $X$;

\item\label{Ll.ii} the composition
\[\Sigma^{r-1} \Lambda(\underline{A})\xrightarrow{\omega_r}T_1\Sigma(\underline{A})\xrightarrow{p} T_1\Sigma(\underline{A}) \mathbin{/} T_{r-1}\Sigma (\underline{A}) \] is trivial,  where $p$ is the quotient map;

\item\label{Ll.iii} there is a retraction \[  T_0\Sigma(\underline{A}) \mathbin{/} T_{r-1}\Sigma (\underline{A}) \xrightarrow{\rho} T_1\Sigma(\underline{A})
\mathbin{/} T_{r-1}\Sigma (\underline{A}) \] for the map \[  T_1\Sigma(\underline{A}) \mathbin{/} T_{r-1}\Sigma (\underline{A}) \xrightarrow{\bar k}
T_0\Sigma(\underline{A}) \mathbin{/} T_{r-1}\Sigma (\underline{A}) \] induced by the inclusion
$k: T_1\Sigma(\underline{A}) \hookrightarrow T_0\Sigma(\underline{A})$.
\end{enumerate}
\end{lemma}

\begin{proposition}\label{wupro} If $r\ge 4$ and $\underline{\mathbb{S}}=(\mathbb{S}^{m_1-1},\dots,\mathbb{S}^{m_r-1})$ with $m_i\ge 1$ for $i=1,\dots,r$
then the quotient map  $T_1\Sigma(\underline{\mathbb{S}})\to T_1\Sigma(\underline{\mathbb{S}})\mathbin{/}T_{r-1}\Sigma(\underline{\mathbb{S}})$ leads
to a non-trivial element of the $r$\textsuperscript{th} order generalized Whitehead product $[0_1,\dots,0_r]$.
\end{proposition}

\begin{proof}\footnote{We are deeply grateful to Jie Wu for his idea on the presented proof.}
Suppose that the quotient map $T_1\Sigma(\underline{\mathbb{S}})\to T_1\Sigma(\underline{\mathbb{S}})\mathbin{/}T_{r-1}\Sigma(\underline{\mathbb{S}})$
yields a trivial element of the $r$\textsuperscript{th} order generalized Whitehead product $[0_1,\dots,0_r]$ for $r\ge 4$.
Then, by Lemma~\ref{Ll}\ref{Ll.iii},  there is a retraction \[  T_0\Sigma(\underline{\mathbb{S}}) \mathbin{/} T_{r-1}\Sigma (\underline{\mathbb{S}})
\xrightarrow{\rho} T_1\Sigma(\underline{\mathbb{S}})
\mathbin{/} T_{r-1}\Sigma (\underline{\mathbb{S}}) \] for the map \[  T_1\Sigma(\underline{\mathbb{S}}) \mathbin{/} T_{r-1}\Sigma (\underline{\mathbb{S}}) \xrightarrow{\bar k}
T_0\Sigma(\underline{A}) \mathbin{/} T_{r-1}\Sigma (\underline{A}) \] induced by the inclusion
$k: T_1\Sigma(\underline{\mathbb{S}}) \hookrightarrow T_0\Sigma(\underline{\mathbb{S}})$.

This leads to homomorphisms of cohomology algebras
\[\rho^\ast : H^\ast(T_1\Sigma(\underline{\mathbb{S}}) \mathbin{/} T_{r-1}\Sigma (\underline{\mathbb{S}}))\to H^\ast(T_0\Sigma(\underline{\mathbb{S}}) \mathbin{/} T_{r-1}\Sigma (\underline{\mathbb{S}}))\] and \[ \bar{k}^\ast : H^\ast(T_0\Sigma(\underline{\mathbb{S}}) \mathbin{/} T_{r-1}\Sigma (\underline{\mathbb{S}}))\to
H^\ast(T_1\Sigma(\underline{\mathbb{S}}) \mathbin{/} T_{r-1}\Sigma (\underline{\mathbb{S}}))\]
over integers with $\bar{k}^\ast\rho^\ast=\mathrm{id}_{H^\ast(T_1\Sigma(\underline{\mathbb{S}})\mathbin{/} T_{r-1}\Sigma (\underline{\mathbb{S}}))}$. In particular, $\rho^\ast$ is monomorphism. The long exact sequence for the triple $T_{r-1}\Sigma (\underline{\mathbb{S}})\subseteq T_1\Sigma(\underline{\mathbb{S}})\subseteq T_0\Sigma(\underline{\mathbb{S}})$ yields homomorphisms \begin{equation}\bar{k}^\ast: H^n(T_0\Sigma(\underline{\mathbb{S}}) \mathbin{/} T_{r-1}\Sigma (\underline{\mathbb{S}}))\to  H^n(T_1\Sigma(\underline{\mathbb{S}}) \mathbin{/} T_{r-1}\Sigma (\underline{\mathbb{S}}))\label{Eq}
\end{equation}which is an isomorphism for $n<m_1+\dots+m_r-1$, and a monomorphism for $n<m_1+\dots+m_r$.

Let now $i :T_{r-1}\Sigma (\underline{\mathbb{S}})\hookrightarrow T_0\Sigma(\underline{\mathbb{S}})$ and $q: T_0\Sigma(\underline{\mathbb{S}})\to T_0\Sigma(\underline{\mathbb{S}}) \mathbin{/} T_{r-1}\Sigma (\underline{\mathbb{S}})$
be the inclusion and quotient maps, respectively. Thus, the K\"unneth Theorem implies that the induced cohomology maps  $i^\ast : H^n(T_0\Sigma (\underline{\mathbb{S}}))\to H^n(T_{r-1}\Sigma(\underline{\mathbb{S}}))$
are splitting epimorphisms for $n\geq 0$, and consequently
\[q^\ast : H^\ast(T_0\Sigma(\underline{\mathbb{S}}) \mathbin{/} T_{r-1}\Sigma (\underline{\mathbb{S}}))\to H^\ast(T_0\Sigma(\underline{\mathbb{S}}))\]
is a monomorphism of cohomology algebras.

Now, consider the non-trivial elements $\alpha\in H^{m_1+m_2}(T_1\Sigma(\underline{\mathbb{S}})\mathbin{/} T_{r-1}\Sigma (\underline{\mathbb{S}}))$ and $\beta\in H^{m_3+\dots+m_r}(T_1\Sigma(\underline{\mathbb{S}}) \mathbin{/} T_{r-1}\Sigma (\underline{\mathbb{S}}))$ determined
by cocycles associated to the cells $e^{m_1}\times e^{m_2}$ and $e^{m_3}\times\dots\times e^{m_r}$, respectively. Then, the cup-product $\alpha\smile\beta\in H^{m_1+\dots+m_r}(T_1\Sigma(\underline{\mathbb{S}}) \mathbin{/} T_{r-1}\Sigma (\underline{\mathbb{S}}))=0$ and consequently, $\rho^\ast(\alpha\smile\beta)=\rho^\ast(\alpha)\smile\rho^\ast(\beta)=0$.

On the other hand, since $\rho^\ast$ and $q^\ast$ are monomorphisms, the elements $q^\ast\rho^\ast(\alpha)$ and $q^\ast\rho^\ast(\beta)$ are non-trivial
as well. Further, in view of (\ref{Eq}), the map \[\bar{k}^\ast: H^n(T_0\Sigma(\underline{\mathbb{S}}) \mathbin{/} T_{r-1}\Sigma (\underline{\mathbb{S}}))\xrightarrow{\cong} H^n(T_1\Sigma(\underline{\mathbb{S}}) \mathbin{/} T_{r-1}\Sigma (\underline{\mathbb{S}}))\] is an isomorphism
with the inverse \[\rho^\ast=(\bar{k}^\ast)^{-1} :  H^n(T_1\Sigma(\underline{\mathbb{S}}) \mathbin{/} T_{r-1}\Sigma (\underline{\mathbb{S}}))\xrightarrow{\cong}
H^n(T_0\Sigma(\underline{\mathbb{S}}) \mathbin{/} T_{r-1}\Sigma (\underline{\mathbb{S}}))\] for $n=m_1+m_2$ and $n=m_3+\dots+m_r$.

Then, the cohomology algebra structure
\[H^\ast(T_0\Sigma(\underline{\mathbb{S}}))\xrightarrow{\cong}H^\ast(\mathbb{S}^{m_1})\otimes\dots\otimes H^\ast(\mathbb{S}^{m_r})\] and $r\ge 4$
lead to the non-trivial cup-product \[q^\ast\rho^\ast(\alpha)\smile q^\ast\rho^\ast(\beta)=q^\ast(\rho^\ast(\alpha)\smile \rho^\ast(\beta))\in H^{m_1+\dots+m_r}(T_0\Sigma(\underline{\mathbb{S}})).\] Consequently, $\rho^\ast(\alpha)\smile\rho^\ast(\beta)\in H^{m_1+\dots+m_r}(T_0\Sigma(\underline{\mathbb{S}}) \mathbin{/} T_{r-1}\Sigma (\underline{\mathbb{S}}))$ is non-trivial as well. This contradiction completes the proof.
\end{proof}

\begin{remark}
The generalized Whitehead map $\omega_r:\Sigma^{r-1}\Lambda(\underline{\mathbb{S}})\to T_1\Sigma(\underline{\mathbb{S}})$ is non-trivial
for any $r\geq 2$. Indeed, if $\omega_r$ is trivial then it factors through the cone $C\Sigma^{r-1}\Lambda(\underline{\mathbb{S}})$.
Hence, in view of \eqref{equ}, there exists a retraction $\rho :  T_0\Sigma(\underline{\mathbb{S}})\to T_1\Sigma(\underline{\mathbb{S}})$
for the inclusion map $k: T_1\Sigma(\underline{\mathbb{S}}) \hookrightarrow T_0\Sigma(\underline{\mathbb{S}})$.

Let $\alpha\in H^{m_1}(T_1\Sigma(\underline{\mathbb{S}}))$ and $\beta\in H^{m_2+\dots+m_r}(T_1\Sigma(\underline{\mathbb{S}}))$ be non-trivial
elements determined by cocycles associated to the cells $e^{m_1}$ and $e^{m_2}\times\dots\times e^{m_r}$, respectively.
Thus, \[\alpha\smile\beta\in H^{m_1+\dots+m_r}(T_1\Sigma(\underline{\mathbb{S}}))=0.\] But,
$\rho^\ast(\alpha),\rho^\ast(\beta)\in H^{\ast}(T_0\Sigma(\underline{\mathbb{S}}))$ satisfy
\[\rho^\ast(\alpha)\smile\rho^\ast(\beta)=\rho^\ast(\alpha\smile\beta)\neq 0 \] and since $\rho^\ast$ is monomorphism, it holds
$\alpha\smile\beta\neq 0$. This contradiction guarantees that $\omega_r$ is non-trivial.
\end{remark}

\medskip

Let $\mathbb{R}$ and $\mathbb{C}$ be the fields of real and complex numbers, respectively and $\mathbb{H}$ the skew $\mathbb{R}$-algebra of quaternions.
Denote by $\mathbb{F}P^n$ the $n$-projective space over $\mathbb{F}=\mathbb{R}$, $\mathbb{C}$ or $\mathbb{H}$, put $d=\dim_\mathbb{R}\mathbb{F}$ and set
$i_{n\mathbb{F}} : \mathbb{S}^d\hookrightarrow\mathbb{F}P^n$ for the inclusion map. Let $\gamma_{n\mathbb{F}}:\mathbb{S}^{(n+1)d-1}\to \mathbb{F}P^n$ be
the canonical quotient map. In view of \cite[Corollary~(7.4)]{barcus} and \cite[(4.1--3)]{BJS}, we obtain a key formula:

\begin{lemma}\label{Y}
Let $h_0f\in\pi_k(\mathbb{S}^{2n-1})$ be the $0$\textsuperscript{th} Hopf--Hilton invariant for $f\in\pi_k(\mathbb{S}^n)$. Then:
\[ [\gamma_{n\mathbb{R}} f,i_{n\mathbb{R}}]
=\begin{cases} 0, & \text{for odd $n$,}\\[1ex]
(-1)^k\gamma_{n\mathbb{R}}(-2f+[\iota_n,\iota_n]\circ h_0f), & \text{for even $n$.}
\end{cases} \]
\end{lemma}

Now, let $f: T_1\Sigma(\underline{\mathbb{S}})\to \mathbb{S}^2$  for $\underline{\mathbb{S}}=(\mathbb{S}^{m_1-1},\dots,\mathbb{S}^{m_r-1})$.
Recall that by \cite[Satz]{baues}, we have $\omega_r(f)=0$ provided $m_1+\cdots+m_r\neq 4$.
Further, Lemma~\ref{Y} implies that $[\pi_{r-1}(\mathbb{R}P^2),i_{2\mathbb{R}}]\neq 0$ in general.
This yields the following result.

\begin{proposition} Let $f: T_1\Sigma(\underline{\mathbb{S}})\to \mathbb{R}P^2$ for $\underline{\mathbb{S}}=(\mathbb{S}^{m_1-1},\dots ,\mathbb{S}^{m_r-1})$.
\begin{enumerate}[label=\textup{(\roman*)}]
\item If $m_i\geq 2$ for $i=1,\dots,r$ with $r\geq 2$, then $\omega_r(f)=0$ provided $m_1+\cdots+m_r\neq 4$.

\item $[\pi_{r-1}(\mathbb{R}P^2),i_{2\mathbb{R}}]\subseteq [f_1,\dots,f_r]$ for $f_i : \mathbb{S}^1\to \mathbb{R}P^2$
with $i=1,\dots,r$ and $f_{i_0}=i_{2\mathbb{R}}$ for at least one $1\le i_0\le r$.
In particular, $2\pi_2(\mathbb{R}P^2) =[0,0,i_{2\mathbb{R}}]=[0,i_{2\mathbb{R}}, i_{2\mathbb{R}}]=[i_{2\mathbb{R}},i_{2\mathbb{R}}, i_{2\mathbb{R}}]$.
\end{enumerate}
\end{proposition}

\begin{proof} \textup{(i):} Since $m_i\geq 2$ for $i=1,\dots,r$, the space $T_1\Sigma(\underline{\mathbb{S}})$ is $1$-connected. Hence, any map
$f:T_1\Sigma(\underline{\mathbb{S}})\to \mathbb{R}P^2$ lifts to $\widetilde{f}: T_1\Sigma(\underline{\mathbb{S}})\to \mathbb{S}^2$ via the quotient
map $\gamma_{2\mathbb{R}}:\mathbb{S}^{2}\to \mathbb{R}P^2$. Consequently, in view of \cite[Satz]{baues},  $\omega_r(f)=0$.

\textup{(ii):} Certainly, in view of Corollary~\ref{cor.fi.zero}, it holds $0\in [f_1,\dots,f_r]$ if at least $f_{i_0}=0$ for some $1\le i_0\le r$.
Further, because the circle $\mathbb{S}^1$ is a topological group, we have $[\iota_1,\overset{\times r}{\dotsc{}},\iota_1]=0$. Hence, by
Theorem~\ref{teo.porter.glue}\ref{teo.porter.2.1}, $0=(i_{2\mathbb{R}})_\ast[\iota_1,\overset{\times r}{\dotsc{}},\iota_1]\in
[i_{2\mathbb{R}},\overset{\times r}{\dotsc{}}, i_{2\mathbb{R}}]$. Then, Corollary~\ref{CC} leads to the inclusion
$[\pi_{r-1}(\mathbb{R}P^2),i_{2\mathbb{R}}]\subseteq [f_1,\ldots, f_r]$ if $f_{i_0}=i_{2\mathbb{R}}$ for at least one $1\le i_0\le r$.

Finally, Lemma~\ref{Y} and \cite[Theorem~0.1]{hardie} lead to $[\pi_2(\mathbb{R}P^2),i_{2\mathbb{R}}]=\linebreak 2\pi_2(\mathbb{R}P^2)=
[0,0,i_{2\mathbb{R}}]=[0,i_{2\mathbb{R}},i_{2\mathbb{R}}]=[i_{2\mathbb{R}},i_{2\mathbb{R}},i_{2\mathbb{R}}]$ and the proof is complete.
\end{proof}

Notice that we can also state:
\begin{remark}
\begin{enumerate}[label=\textup{(\roman*)}]\thmenumhspace
\item $0\in [i_{n\mathbb{R}},\overset{\times r}{\dotsc{}}, i_{n\mathbb{R}}]\subseteq \pi_{r-1}(\mathbb{R}P^n)$ for $r\ge 2$. In particular,
$[i_{n\mathbb{R}},\overset{\times r}{\dotsc{}}, i_{n\mathbb{R}}]=0$ for $r\leq n$ and $2\pi_{r-1}(\mathbb{R}P^{2n})\subseteq [i_{2n\mathbb{R}},\overset{\times r}{\dotsc{}}, i_{2n\mathbb{R}}]\subseteq \pi_{r-1}(\mathbb{R}P^{2n})$ for $r<4n$.

\item  $[i_{r\mathbb{C}},\overset{\times (r+1)}{\dots\dotsc{}}, i_{r\mathbb{C}}]=(r+1)! \gamma_{r\mathbb{C}}$ for $r\ge 2$ (\cite[Corollary 2]{porter2}).

\item It is known that $[i_{1\mathbb{H}},i_{1\mathbb{H}}]=[\iota_4,\iota_4]=\pm(2\nu_4-\Sigma\nu')$ and
$[i_{n\mathbb{H}},i_{n\mathbb{H}}]=(i_{n\mathbb{H}})_{\ast}[\iota_4,\iota_4]=\pm i_{n\mathbb{H}}(\Sigma\nu')\neq 0$ for $n\ge 2$. Hence, $[i_{n\mathbb{H}},\overset{\times r}{\dotsc{}},i_{n\mathbb{H}}]=\emptyset$ for $r\ge 3$ and $n\ge 1$.
\end{enumerate}
\end{remark}

\end{document}